\documentclass[10pt,reqno]{amsart}

\usepackage[english]{babel}
\usepackage[utf8]{inputenc}
\usepackage{ae}

\usepackage{eucal}
\usepackage[dvips]{graphicx}
\usepackage{epsfig}
\usepackage{amsfonts,amsthm}
\usepackage{amssymb}
\usepackage[all]{xy}
\usepackage{xcolor}

\newtheorem{problem}{\sc Problem}

\newtheorem{lemma}{\sc Lemma}
\newtheorem{theorem}{\sc Theorem}

\newtheorem{ex}{\sc Example}

\newtheorem{rk}{\sc Remark}

\begin{document}

\title[New solutions to  Euler equations]
      {2d incompressible Euler equations: new explicit solutions}

\author[M. J. Mart\'{\i}n]{Mar\'{\i}a J. Mart\'{\i}n}
\address{Departamento de Matem\'aticas, Facultad de Ciencias (M\'odulo 17), Universidad Aut\'o\-no\-ma de Madrid. Campus de Cantoblanco, 28049, Madrid. Spain.}\email{mariaj.martin@uam.es}

\author[J. Tuomela]{Jukka Tuomela}
\address{Department of Physics and Mathematics, University of Eastern Finland. P.O. Box 111, FI-80101 Joensuu, Finland.} \email{jukka.tuomela@uef.fi}

\thanks{The first author is supported by UAM and EU funding through the InterTalentum Programme (COFUND 713366). She also thankfully acknowledges partial support from Spanish MINECO/FEDER research project MTM2015-65792-P}

\date{\today}



\begin{abstract}

There are not too many known explicit solutions to the $2$-dimens\-ion\-al incompressible Euler equations in Lagrangian coordinates. Special mention must be made of  the well-known ones due Gerstner and Kirchhoff, which were already discovered in the  $19$th century.  These two classical solutions share a common characteristic, namely,  the dependence of the coordinates from the initial location is determined by a harmonic map, as recognized by Abrashkin and Yakubovich, who more recently -in the $1980$s- obtained new explicit solutions with a similar feature.
 
We present a more general method for constructing new explicit solutions in Lagrangian coordinates which contain as special cases all previously known ones. This new approach shows that in fact ``harmonic labelings'' are special cases of a much larger family. 

In the classical solutions, the matrix Lie groups were essential in describing the time evolution. We see that also the geodesics in these groups are important.

\end{abstract}

\maketitle

\section{Introduction}

There are two standard frameworks for analyzing the motion of an ideal homogeneous fluid. The Eulerian description observes at fixed locations the flow properties as the particles go by. The motion is then obtained by imposing the law of mass conservation 
\begin{equation}\label{eq-massconservation}
 \nabla \cdot u =0
\end{equation}
and the conservation of momentum law
\begin{equation}\label{eq-Euler}
 u_t+u\nabla u +\nabla_{\!x} p=0\,,
\end{equation}
where $u=u(t,x)$ is the velocity field in the time $t$ and space $x=(x_1, \ldots, x_n)$ variables, the scalar function $p$ represents the pressure and $\nabla_{\!x}$ denotes the gradient with respect to space variables.

An alternative representation of the flow is provided by the (material) Lagrangian coordinates, in which the observer follows the fluid by picking out a particular particle and keeping track of where it goes. There exists a precise Eulerian state corresponding to a Lagrangian state and vice versa (see \cite{C}). 
\par\smallskip
Let us briefly review  how the equations of the motion look like in Lagrangian coordinates. 
\par
Starting with a domain $\Omega_0\subset \mathbb{R}^n$ and using $x$ to denote the Eulerian (spatial) coordinates and $\beta=(\beta_1, \ldots, \beta_n)$ to denote the Lagrangian coordinates, we see that each coordinate $\beta\in\Omega_0$ identifies by means of the  map
\begin{equation}\label{eq-map}
\beta \mapsto x( t, \beta)=\Phi( t, \beta)=\Phi^t(\beta)
\end{equation}
the evolution in time of a specific particle. Here $\Phi^t$ is a diffeomorphism, $\Omega_t=\Phi^t(\Omega_0)$, and  $\Phi^0$ is the identity. Within these terms the Eulerian and Lagrangian descriptions of the flow are related by the equation
\begin{equation}\label{eq-eq}
\Phi'(t, \beta)=\frac{\partial}{\partial t} \Phi(t, \beta)=u(t,\Phi(t, \beta))\,.
\end{equation}
The reader may be advised of the fact that, unless otherwise explicitly stated, this is the notation we will use in this paper; namely, the derivative of any map $f$ with respect to $t$ will be denoted by $f'$.
\par
The incompressibility condition (\ref{eq-massconservation}) gives then the equation $\det(d\Phi^t)=1$ for all $t$. 
\par 
It is easy to check that differentiating (\ref{eq-eq}) with respect to $t$ and using (\ref{eq-Euler}) we have
\begin{equation}\label{eq-prelim}
 \Phi''(t, \Phi(t, \beta))=\left(u_t+u\nabla u\right)(t,\Phi(t,\beta))=-\nabla_{\!x} p(t, \Phi(t, \beta))\,.
\end{equation}

A direct application of the chain rule shows that the relation between the gradient  $\nabla_x p$ as above and $\nabla_{\!\beta} p$ (now with respect to the Lagrangian coordinates $\beta$) is given by $\nabla_{\!\beta}= (d\Phi^t)^T \nabla_{\!x}$. Therefore, we obtain that (\ref{eq-prelim}) is equivalent to
\begin{equation}\label{eq-prelim2}
       (d\Phi^t)^T \Phi''+\nabla_{\!\beta} p=0\,.
\end{equation}
\par
Finding explicit solutions in this form, when the initial mapping $\Phi^0$ satisfies $\det(d\Phi^0)=1$, is unnecessarily hard (\emph{cf.} Section~\ref{sec-non-existence} below). However, it is possible to introduce a further modification as follows. Let $D\subset \mathbb{R}^n$ be a domain and assume that there is a diffeomorphism $\hat\varphi\,:\,D\to \Omega_0$. We call $D$ the labelling domain or parameter domain, and denote the  labels, \emph{i.e.} the coordinates of $D$, by  $\alpha=(\alpha_1, \ldots, \alpha_n)$.  Within these terms, another application of the chain rule allows us to re-write  \eqref{eq-prelim2} in terms of these new coordinates as
\[
(d\varphi^t)^T\varphi''+\nabla_{\!\alpha} p=0\,.
\]
In order not to burden the notation, let us agree with the convention that, otherwise specified, we use $\nabla$ to denote $\nabla_\alpha$ and $(d\varphi^t)^T=(d\varphi)^T$. We can then look for the solutions to the following problem.

\begin{problem}\label{problem}
Find diffeomorphisms $\varphi=\varphi(t,\alpha)=\varphi^t(\alpha)$ and a function $p=p(t,\alpha)$ with $\alpha \in D$ and $t\geq 0$ such that
\begin{equation}
       (d\varphi)^T\varphi''+\nabla p=0
\label{yht}
\end{equation}
and with $\det(d\varphi^t)=\det(d\varphi^0)\ne 0$ for all such $t$.
\end{problem}
\par\smallskip

If such a $\varphi$ can be found then we can define $\Phi^t=\varphi^t\circ \hat\varphi^{-1}$  which then gives the Lagrangian description of the fluid properly speaking. 
\par\smallskip

\begin{rk}
Note that there is no need to suppose that the domain is simply connected. This is because the existence of the pressure is a consequence of the fact that in Eulerian coordinates the velocity field is divergence free. In fact, one of the standard problems in the applications of fluid mechanics is the two dimensional flow of the air around the airplane wing and in this case the domain is not simply connected. 
\end{rk}

Though the non-linearity of the problem makes it rather difficult to construct interesting explicit solutions in either Eulerian coordinates or Lagrangian coordinates, it seems that the most complete description of a flow is attained within the Lagrangian framework. Some of the celebrated explicit solutions to the two-dimensional incompressible Euler equations in Lagrangian variables are in fact quite old: Gerstner's flow \cite{G},  found in 1809 (rediscovered in 1863 by Rankine \cite{R});  Kirchhoff's elliptical vortex \cite{K}, found in 1876. More recently Abrashkin and Yakubovich \cite{AY} found some new solutions in 1984. Somewhat surprisingly, apparently these examples (and small variations of them) were all known explicit solutions up to short time ago. 
\par

In the construction of all these classical flows, harmonic maps are essential, as all of them present a labelling by harmonic functions. A. Aleman and A.~Constantin \cite{AC} proposed a complex analysis approach aimed at classifying all such flows. With the aim of complementing the work in \cite{AC}, a new different approach, based on ideas from the theory of harmonic mappings, was used in \cite{CM}, where the authors explicitly provide all solutions, with the specified structural property, to the incompressible 2-dimensional Euler equations (in Lagrangian variables). 
\par
It is not difficult to check that the map (\ref{eq-map}) corresponding to the classical solutions due to Gerstner and Kirchhoff (and those ones obtained by Abrashkin and Yakubovich in \cite{AY}) as well as those in \cite{AC} and \cite{CM} are as follows. Up to an additive constant, either
\begin{equation}\label{eq-structure-affine}
\varphi(t, \alpha)=\varphi^t(\alpha)=A(t)v(\alpha)\,, 
\end{equation}
where $A$ belongs to the special linear group $\mathbb{SL}(2)$ of $2\times 2$ matrices with determinant $1$ and $v$ is a vector field whose coordinates are harmonic functions and such that $\det(dv)\neq 0$, or
\begin{equation}\label{eq-structure-affine-combination}
\varphi(t, \alpha)=\varphi^t(\alpha)=M_1(t)v(\alpha)+M_2(t)w(\alpha)\,,
\end{equation}
where $M_1$, $M_2$ belong to the group $\mathbb{O}(2)$ of $2\times 2$ orthogonal matrices and $v$ and $w$ are, again, vector fields whose coordinates are harmonic functions and satisfy $\det(dv)\ne 0$ and $\det(dw)\ne 0$. 
\par\smallskip
In this article, instead of considering those solutions to the $2$-dimensional incompressible Euler equations for which the map (\ref{eq-map})  is harmonic for all times $t$, we will focus on analyzing those solutions for which the labelling map $\varphi$ takes one of the forms described by (\ref{eq-structure-affine}) or (\ref{eq-structure-affine-combination}) without any assumption on harmonicity of the vector fields involved in this description. The methods used show that perhaps curiously the harmonicity of the maps is in fact not essential (\emph{cf.} Sections~\ref{sec-sol1} and \ref{sec-sol2} below). Harmonic maps simply provide one family of solutions to a certain PDE system which has plenty of other solutions as well. This PDE system allows us to  construct new families of solutions. 

 \section{Preliminaries and notation}
Unfortunately, for questions of space, it is not possible to include all the details related to the theory involved in the approach developed in this paper to obtain new explicit solutions to the $2$-dimensional incompressible Euler equations, which is the main goal in this article. Nevertheless, with the hope to make this paper self-contained, we now review the main tools, concepts, and results used in our standpoint. We also include the main references to help the reader to figure out the key points in the proofs of our results, though we should point out that it is possible to check directly  that the functions $\varphi$ obtained in our main Theorems~\ref{thm-pde} and \ref{thm-pde2} -or in Theorems~\ref{thm-2d} and \ref{thm-2d-2}- satisfy the requirements stated in Problem~\ref{problem}, thus they provide new explicit solutions to the problem considered.
\subsection{Geometry}
We will need some elementary notions related to Riemannian geometry; one standard reference is \cite{petersen}. 
\par
Recall that given a Riemannian manifold $M$, the curve $a\,:\,\mathbb{R}\to M$ is a \emph{geodesic} if it satisfies the differential equation
\begin{equation}\label{eq-Christoffel}
        (a'')^k+\Gamma_{ij}^k  (a')^i (a')^j=0\,.
\end{equation}
Here we are using the \emph{Einstein summation convention}: in those cases when in a single term an index appears twice (once up and once down) and is not otherwise defined, it implies summation of that term over all the values of the index.  The symbols $\Gamma_{ij}^k$ are the Christoffel symbols of second kind and, as above, $a'$ represents the derivative of the curve $a$ with respect to the parameter (the time $t$, say) it depends on. 
\par
In those cases when $M$ is an $n-1$ dimensional submanifold of $\mathbb{R}^n$ with the induced Riemannian metric, it is sometimes convenient to express the geodesic equations in ambient coordinates. More concretely, let $b$ be any non-zero normal vector field of $M$. Then $a\,:\,\mathbb{R}\to \mathbb{R}^n$ is a geodesic if there is some function $\lambda\,:\,M\to\mathbb{R} $ such that
\[
  \begin{cases}
     a''+\lambda b=0\,,\\
     \langle  a',b\rangle=0\,,\\
     a(t)\in M\,.
     \end{cases}
\]

\subsection{Algebra}  In this section, we refer the reader to  \cite{colios} for more information about questions related to computational commutative algebra. Though  the theory about the relationship between varieties and ideals is much more developed in the complex case than in the real case, in our context only real varieties are of interest. For general information about real algebraic geometry, see \cite{bocoro}. 
\par\smallskip
Let us consider polynomials of variables $x_1,\dots,x_n$ with coefficients in the field $\mathbb{K}$ and let us denote the ring of all such polynomials by $\mathbb{A}=\mathbb{K}[x_1,\dots,x_n]$. The given polynomials $f_{1},\ldots ,f_{k}\in \mathbb{A}$ generate an \emph{ideal}

\begin{equation*}
I =\langle f_{1},\ldots ,f_{k}\rangle =\{f\in \mathbb{A}
\colon f=h_{1}f_{1}+\ldots +h_{k}f_{k},~\text{where}~h_{i}\in \mathbb{A}\}.
\end{equation*}
We say that the polynomials $f_{i}$ are \emph{generators} of the ideal $I$ and as a set they are the \emph{basis} of $I$.  Notice that we do not assume that the polynomials $\{ f_{1},\dots,f_k\}$ are  independent in any way so that if $f_{k+1}\in I $, then $\{ f_{1},\dots, f_{k}, f_{k+1}\}$ is also a basis of $I $. 
\par\smallskip
\par
Not all the bases of an ideal are equally good; the good ones are known as \emph{Gr\"obner bases}. Gr\"obner bases depend on the monomial order, but once we have chosen a particular order the Gr\"obner basis is essentially unique and one can actually compute it with the  \emph{Buchberger algorithm}.   It may be convenient as well to stress that in actual computations below (related to Gr\"obner bases), we use {\sc Singular}~\cite{singular}. 

Let $I \subset\mathbb{A}$ and let $G=\{g_1,\dots,g_\ell\}$ be a Gr\"obner basis of $I$. Then given any $f\in\mathbb{A}$ we can compute the representation $f=\sum_{j=1}^\ell a_jg_j+R$ such that  $R$ (the remainder) is unique. This remainder is also known as the \emph{normal form} of $f$ with respect to $I$ and in this case we can write $R=\mathsf{NF}(f,I)$ or  $R=\mathsf{NF}(f,G)$. The normal form can also be interpreted as an element of the \emph{residue class ring} $\mathbb{A}/I$. 
\par
The \emph{radical} of $I $ is
\begin{equation*}
\sqrt{I }=\{f\in \mathbb{A}\colon f^{m}\in I\ {\rm for\ some\ integer\ } m\geq 1\}.
\end{equation*}
An ideal $I $ is a \emph{radical ideal} if $I =\sqrt{I }$.  The \emph{real radical}  $\sqrt[\mathbb{R}]{I }$ is the set of polynomials $f\in \mathbb{A}$ which satisfy the condition that 

\[
f^{2r}+\sum_{i=1}^m g_i^2\in I 
\]
for some positive integers $r$ and $m$ and some functions $g_i,\ i=1, \ldots, m$, in $\mathbb{A}$.
An ideal is a \emph{real ideal} if $\sqrt[\mathbb{R}]{I }=I $. In particular, if $I $ is a real ideal and 
\[
f_1^2+\dots+f_m^2\in I\,,
\]
then $f_j\in I $ for all $j$. 

An ideal $I $ is \emph{prime} if the condition that  $fg\in I$ implies that, necessarily, either $f\in I$ or $g\in I$. Every prime ideal is a radical ideal.
\par\smallskip
Let $I \subset\mathbb{A}$. Then $k$-\emph{th elimination ideal} of $I$ equals
\begin{equation*}
I _k=I \cap \mathbb{K}[x_{k+1},\dots,x_n].
\end{equation*}

Geometrically, the concept of $k$-th elimination ideal is related to the \emph{projections} $\pi_k\,: \mathbb{L}^n\to \mathbb{L}^{n-k}$ defined by
\[
\pi_k\left((x_1,\dots,x_n)\right) =(x_{k+1},\dots,x_n).
\]
A Gr\"obner basis of an elimination ideal can also be computed using a suitable product order in the Buchberger algorithm.
\par
To each ideal $I $ we can associate the corresponding variety $\mathsf{V}(I)$. There are various ways to define the associated variety depending on the desired level of abstraction. For us, the following procedure is the most convenient.
\par
Let $\mathbb{L}$ be some extension field of $\mathbb{K}$ (typically, in applications, one has $\mathbb{K}=\mathbb{Q}$ and $\mathbb{L}=\mathbb{R}$). Then we set 
\[
\mathsf{V}(I )=\big\{a\in \mathbb{L}^{n}\colon f(a)=0\ {\rm for\ all\ }
~f\in I \big\}\subset \mathbb{L}^{n}.
\]
Note that $\mathsf{V}(I )=\mathsf{V}(\sqrt{I })$.
\par
The connection between all these notions is that if $\mathbb{L}$ is algebraically closed then 
\begin{equation*}
\overline{\pi_k\big(\mathsf{V}(I )\big)}=\mathsf{V}(I _k)\,,
\label{projektio}
\end{equation*}
where the overline denotes the \emph{Zariski closure}.
\par\medskip

\subsection{PDEs}
Let  $\nu$ be a multi-index and, as usual, let  $|\nu|=\nu_1+\dots+\nu_n$. It is well-known that any linear PDE can be written as
\[
        Lu=\sum_{|\nu|\le q}c_\nu\partial^\nu u=f\,,
\]
where $c_\nu$ are some known matrices, not necessarily square. 
\par
The \emph{principal symbol} $\sigma L$ of the operator $L$ is defined by
\begin{equation*}\label{eq-elliptic}
     \sigma L=\sum_{|\nu|= q}c_\nu\xi^\nu\,,\quad \xi\in\mathbb{C}^n\,,
\end{equation*}
and we say that $L$ is \emph{elliptic} if $\sigma L$ is injective for all $\xi\in\mathbb{R}^n$, $\xi\neq 0$.
\par\smallskip
In the context of formal theory of PDEs \cite{werner} it is convenient to use the so-called \emph{jet notation for derivatives}. More concretely, given a vector field $u=(u^1,\dots,u^n)$, we use 
\[
    u_\nu^j=\partial^\nu u^j=
    \frac{\partial^{|\nu|}u^j}{\partial x_1^{\nu_1}\cdots\partial x_n^{\nu_n}}
\]
to denote the derivatives of the components of $u$. Also, we denote by $u_{|\nu|}$ the vector which contains all derivatives of $u$ of order $|\nu|$.

One final remark must be mentioned at this point. There are two natural conventions regarding the sign in the Cauchy-Riemann equations. Namely, one which is used in complex analysis and the other which comes from the de Rham complex. In the present context, it is more natural to use the one provided from the latter theory which allows us to use the same definition in any dimension. Hence, we say that  a map or vector field $u$ is a \emph{Cauchy-Riemann map} (or just CR map) if $du$ is symmetric and $\mathsf{tr}(du)=\mathsf{div}(u)=0$.

\section{Area-preserving and harmonic solutions}\label{sec-non-existence}

The aim of this section is to show that under the assumption that the maps considered in Problem~\ref{problem} are area-preserving (so that $\det(d\varphi^t)=1$ for all $t\geq 0$) and harmonic in the domain $D$  is in a sense too restrictive, as Theorem~\ref{thm-non-existence} below shows. We have not been able to find an explicit reference for this result, which we  include for the sake of completeness. In our proof, we use the formal theory of PDEs. A comprehensive outline of the relevant concepts can be found in the monographs  \cite{pommaret} and \cite{werner}.  We refer the reader to these books for the details.
\par\smallskip
Let $\mathcal{E}=\mathbb{R}^n\times \mathbb{R}^n$ be a bundle with coordinates $(x,y)$ and let the projection $\pi\,:\,\mathcal{E}\to \mathbb{R}^n$ be defined by $\pi(x,y)=x$. The corresponding \emph{jet bundle of order} $q$ is denoted by $J_q(\mathcal{E})$.  
\par
Let a system of PDEs of order $q$ be given by a map $F\,:\,J_q(\mathcal{E})\to \mathbb{R}^k$. To this map we can associate a set $\mathcal{R}_q=F^{-1}(0)\subset J_q(\mathcal{E})$. Two fundamental operations come up naturally: projection and prolongation.

To \emph{prolong} a given system of PDEs, one just differentiates all the equations. More concretely, if we have system $\mathcal{R}_q\subset J_q(\mathcal{E})$ given by some map $F$, then its first prolongation is $\mathcal{R}_{q+1}=\rho_1(\mathcal{R}_q)\subset J_{q+1}(\mathcal{E})$, given by
\[
 \mathcal{R}_{q+1}=\rho_1(\mathcal{R}_q)\quad\equiv\quad
 \begin{cases}
    F=0\,,\\
    \frac{\partial F}{\partial x_j}=0, & \ 1\le j\le n\,.
    \end{cases}
\]

As an explicit example that shows how this procedure works, let us consider the \emph{Killing equations} in $\mathbb{R}^n$ (recall that Killing vector fields are infinitesimal generators of the volume preserving maps). These equations are given by the system
\[
 \mathcal{R}_1\quad\equiv\quad  \left\{\partial_j u^i+\partial_i u^j=0\,,\right.\quad 1\leq i, j\leq n\,.
 \]

In order to get the prolongation $\mathcal{R}_2$, we differentiate all equations in $\mathcal{R}_1$ to obtain 
\[
\partial^2_{jk}u^i+\partial^2_{ik}u^j=0\,,\quad  1\leq i, j, k\leq n\,.
\]
Now, it is obvious that we can interchange the values of $i$, $j$, and $k$ to get

\[
  \begin{cases}
   \partial^2_{jk}u^i+\partial^2_{ik}u^j=0\,,\\
    \partial^2_{kj}u^i+\partial^2_{ij}u^k=0\,,\\
      \partial^2_{ji}u^k+\partial^2_{ki}u^j=0\,,
       \end{cases}
\]
which gives (summing up the first two equations and using the third one) $2\partial^2_{jk}u^i+\partial^2_{ik}u^j+\partial^2_{ij}u^k=0$. That is, $\partial^2_{jk}u^i=0$ for all $1\leq i, j, k\leq n$.
Hence,
\[
    \mathcal{R}_2\quad\equiv\quad  \begin{cases}
     \partial_j u^i+\partial_i u^j=0\,,\\
     \partial^2_{jk}u^i=0\,,
     \end{cases}
\]
and therefore, we obtain that the solutions are $u(x)=Ax+b$, where $A$ is a skew symmetric matrix; the dimension of the solution space is $n(n+1)/2$.
\par
\begin{rk}
The Killing equations are of \emph{finite type}: after prolonging ``far enough''  (just once in this case) it is possible to express all partial derivatives of certain order in terms of lower order derivatives and the formal solution space is finite dimensional, so that there is only a finite number of degrees of freedom in the general solution.
\end{rk}
\par\smallskip
Usually, in addition to prolongation, one needs to ``project''. The \emph{projections} $\pi^{q+r}_q\,:\,J_{q+r}(\mathcal{E})\to J_q(\mathcal{E})$ 
are simply the maps which forget the highest derivatives or jet variables.   When this map is restricted to the differential equation, we obtain the map $\pi_q^{q+r}\,:\,\mathcal{R}_{q+r}\rightarrow \mathcal{R}_q$. The image of this map is denoted by $\mathcal{R}_q^{(r)}$. Note that we always have $\mathcal{R}_q^{(r)}\subset \mathcal{R}_q$. The fact that the inclusion is strict means that by differentiating and eliminating, we have found \emph{integrability conditions}; that is, equations of order $q$ which are algebraically independent of the original equations and which are also satisfied by the solutions of the system. 
\par
The goal of this type of analysis is thus to find all integrability conditions. The theory behind this intuitive idea is too involved to be developed in this paper; as mentioned above, the details can be found in \cite{pommaret} and \cite{werner}. For the result given below the general theory is not needed. Note that when one finds integrability conditions, the formal solution space remains the same. 
\par
Now we state the main result in this section.
\par\smallskip
\begin{theorem}\label{thm-non-existence}
Let $y=(y^1, y^2)\,:\,\mathbb{R}^2\to \mathbb{R}^2$ be a smooth area-preserving map. Suppose that  $y$ is harmonic, that is, $ \Delta y^1=\Delta y^2=0$. Then $y$ is an affine map.
\end{theorem}
\begin{proof}
Let us define
\[
  f_1=   \det(dy)-1=y^1_{10}y^2_{01}-y^1_{01}y^2_{10}-1\,,
 \] 
where, as usual, the first sub-index denotes the number of derivatives with respect to the first variable and the second sub-index denotes the number of derivatives with respect to the second variable.
\par
The requirement that  $y$ is area preserving ($f_1=0$) and harmonic produces the initial system
\begin{equation}\label{eq-system}
\mathcal{R}_2\quad\equiv\quad    \begin{cases}
   f_1=y^1_{10}y^2_{01}-y^1_{01}y^2_{10}-1=0\,,\\
   f_2=y^1_{20}y^2_{01}+y^1_{10}y^2_{11}-y^1_{11}y^2_{10}-y^1_{01}y^2_{20}=0\,,\\
   f_3=y^1_{11}y^2_{01}+y^1_{10}y^2_{02}-y^1_{02}y^2_{10}-y^1_{01}y^2_{11}=0\,,\\
   f_4=y^1_{20}+ y^1_{02}=0\,,\\
   f_5=y^2_{20}+ y^2_{02}=0\,,
   \end{cases}
\end{equation}
where $f_2$ and $f_3$ are obtained by prolonging $f_1$. Hence we can write $\mathcal{R}_2=f^{-1}(0)\subset J_2(\mathcal{E})$,  where $f=(f_1,f_2,f_3,f_4,f_5)$. 
\par
In order to prove the result, we need to check that all solutions of the system  (\ref{eq-system}) are, in fact, affine.
Notice that the equations in \eqref{eq-system} are polynomials in jet variables so that we can write $\mathcal{R}_2\simeq \ \mathbb{R}^2\times \mathbb{R}^2\times V_2$, where  $V_2=\mathsf{V}(I_2)\subset\mathbb{R}^{10}$. Here, $I_2=\langle f_1,f_2,f_3,f_4,f_5\rangle\subset \mathbb{Q}[y_1,y_2]$, where, as explained in the previous section, we use $y_{|\nu|}$ to denote the vector which contains all derivatives of $y$ of order $|\nu|$.
 
\par
The prolongation can then be written as
\begin{align*}
 \mathcal{R}_3\quad\equiv\quad  
 \begin{cases}
 \mathcal{R}_2\,,\\
  \partial^\nu f_1=0\,,&\ |\nu|=2\,,\\
    \partial^\nu f_4=0\,,&\ |\nu|=1\,,\\
     \partial^\nu f_5=0\,,&\ |\nu|=1\,.
   \end{cases}
\end{align*}

There are thus 7 equations  ($f_i=0$, $i=6, 7, \ldots, 12$) of order 3. The relevant ideal can be written as
\[
   I_3=I_2+\langle f_6,\dots,f_{12}\rangle\subset \mathbb{Q}[y_1,y_2,y_3]\,.
\]
To construct the projection we must then compute the $8$th elimination ideal $I_{3,8}$ of $I_3$ which is
\begin{equation*}
       I_{3,8}=I_3\cap \mathbb{Q}[y_1,y_2]\,.
\end{equation*}
Now choosing an appropriate product order, we obtain the Gr\"obner basis 
\[
\{f_1,\dots,f_5, g_1,g_2,g_3,g_4\}
\]
 of $I_{3,8}$, where

\begin{equation*}
\begin{array}{ccccc}
    g_1= y^1_{11}y^2_{20}-y^2_{11}y^1_{20},  &   &  &  &   g_2=  (y^1_{11})^2+(y^1_{20})^2,\\     
    g_3= y^2_{11}y^1_{11}+y^2_{20}y^1_{20}, &  & {\rm and} &  & g_4=   (y^2_{11})^2+(y^2_{20})^2\,.
   \end{array}
\end{equation*}
However, since we are only interested in real varieties, we get (from the conditions $g_2=g_4=0$) that, necessarily,
\[
  y^1_{11}=y^1_{20}=y^2_{20}=y^2_{11}=0\,.
\]
Hence all second derivatives are zero and we have
\[
\mathcal{R}_2^{(1)}\quad\equiv\quad  
  \begin{cases}
    y^1_{10}y^2_{01}-y^1_{01}y^2_{10}-1=0\,,\\
   y^k_\nu=0\,,&\ |\nu|=2\,, k=1,2\,.
  \end{cases}
\]
These conditions imply that $y=y(x)=Ax+b$  for some constant matrix $A$ with determinant equal to $1$ and some constant vector $b$. 
\end{proof}

\begin{rk}
Note that in the proof of the previous theorem, we have essentially constructed the real radical of the elimination ideal $I_{3,8}$ which coincides with
\[
   \sqrt[\mathbb{R}]{I_{3,8}}=\langle y^1_{10}y^2_{01}-y^1_{01}y^2_{10}-1\rangle+
   \langle y^k_\nu, |\nu|=2\,, k=1,2\rangle\,.
\]
In general, computing the real radical is actually very difficult, and searching  for good algorithms is an active research topic \cite{sdp-sos}. This shows that constructing the projection is  sometimes quite tricky.
\end{rk}
\par
\begin{rk}
The system  $\mathcal{R}_2$, defined by (\ref{eq-system}), is of \emph{finite type}, like the system of Killing equations. Here in $\mathcal{R}_2^{(1)}$ all second order derivatives are given, and indeed there are only 5 arbitrary constants in the general solution. This number comes directly from our computations if we consider $\sqrt[\mathbb{R}]{I_{3,8} }$ as an ideal in the ring $\mathbb{Q}[y,y_1,y_2]$ since in this case 
\[
  \mathsf{dim}\big(\mathsf{V}(\sqrt[\mathbb{R}]{I_{3,8} })\big)=5\,.
\]
The dimension of the variety can be computed once the Gr\"obner basis is available.

Note that the dimension of the variety $\mathsf{V}(I_{3,8})$ considered in the proof of Theorem~\ref{thm-non-existence} as a complex variety is higher than its dimension  as a real variety. More concretely,
\[
  \mathsf{dim}\big(\mathsf{V} ( I_{3,8})\big)=6>5=
   \mathsf{dim}\big(\mathsf{V}(\sqrt[\mathbb{R}]{I_{3,8} })\big)\,.
\]
\end{rk}

\section{New solutions to the $2$D incompressible Euler equations, I}\label{sec-sol1}
As mentioned in the introduction, in this paper instead of considering those solutions to the $2$-dimensional incompressible Euler equations for which the map (\ref{eq-map})  is harmonic for all times $t$, we will focus on analyzing those solutions for which the labelling map $\varphi$ takes one of the forms described by (\ref{eq-structure-affine}) or (\ref{eq-structure-affine-combination}) without any assumption on harmonicity of the vector fields involved in this description. 
\par
We now consider the cases when $\varphi$ is given by (\ref{eq-structure-affine}).  Our first result is the following lemma, which works in any dimension $n$.

\begin{lemma}\label{lem-symmetric} Let $A\in\mathbb{SL}(n)$. Then the map of the form (\ref{eq-structure-affine})  provides a solution to the incompressible Euler equations if $A^TA''$ is symmetric and $v$ is any diffeomorphism. In this case the pressure is given by 
\[
  p=-\tfrac{1}{2}\,\langle v,A^TA''v\rangle\,.
\]
\end{lemma}
\begin{proof}
It is clear that $\det(d\varphi^t)=\det(d\varphi^0)\neq 0$. A straightforward calculation shows that (\ref{yht}) holds. 
\end{proof}
\par
\begin{rk}
Notice that the assumption on the symmetry of  $A^TA''$ in the previous lemma is rather a mild condition since the dimension of $\mathbb{SL}(n)$ is $n^2-1$ and the condition on the symmetry $(A^TA'')^T=A^TA''$ gives $\tfrac{1}{2}\,n(n-1)$ differential equations of second order. So when $n=2$ we expect that there are 2 arbitrary functions in the solution and this is precisely what happens, as will become apparent below.
\end{rk}

\par
It may seem strange that there is no condition on the vector field $v$ in the previous lemma, except that it is a diffeomorphism. Let us outline the reasons that justify this observation.  
\par\smallskip
Recall that the labelling coordinates are completely arbitrary, they do not have any physical significance. Let now $\varphi$ have the form  (\ref{eq-structure-affine}). Then, by writing
\[
\varphi(t,\alpha)=\left(A(t)A^{-1}(0)\right)\,\left(A(0)v(\alpha)\right)=\widetilde{A}(t) V(\alpha)\,,
\]
if needed, we can assume without loss of generality that $A(0)=I$ and $v$ is a diffeomorphism from the labelling domain $D$ to the Lagrangian domain $\Omega_0$. 
\par
We then have from (\ref{eq-map}) that we can write
\begin{equation}\label{eq-Eulerian}
x=\Phi^t(\beta)=\varphi^t\circ v^{-1}=A(t)\beta\,,
\end{equation}
 so that the Eulerian coordinates depend linearly on the Lagrangian coordinates, and since the map $v$ does not appear anymore in the physical description of the flow when considering the Eulerian coordinates, it makes sense that there is no need to have conditions on it. 

Let us also see how the motion looks like in Eulerian coordinates in this case. Using (\ref{eq-Eulerian}), it is obvious that 
\[
   A'(t)\beta=x'(t)=u(t, x(t))=u(t, A(t)\beta)\,,
 \] 
so that $u(t, x)=A' A^{-1}x$. That is, the velocity vector field is a linear vector field. 
\begin{rk}
Since we are assuming that $A\in\mathbb{SL}(n)$, the matrix that determines the velocity vector field, $A'A^{-1}$, has the property $\mathsf{trace}(A'A^{-1})=0$. Hence, $A'A^{-1}$ is an element of the Lie algebra $\mathfrak{sl}(n)$ corresponding to the Lie group $\mathbb{SL}(n)$. If we take a smaller group and require that $A\in\mathbb{SO}(n)$ then $A'A^{-1}=A'A^T$ is skew symmetric, \emph{i.e.} an element of $ \mathfrak{so}(n)$. Moreover, now we easily obtain 
\[
   u_t+u\nabla u=A''A^{-1}x=-\nabla_x p\,.
\]
\par
Note also  that  $A^TA''$ is symmetric implies that $A''A^{-1}$ is symmetric so that the pressure is given by
\[
p=-\tfrac{1}{2}\,\langle \beta,A^TA''\beta\rangle=-\tfrac{1}{2}\,\langle x,A''A^{-1}x\rangle\ .
\]
\end{rk}

\par\smallskip
Now, let us focus on the $2$-dimensional case. In order to present the results, it is convenient to stress that any matrix $A\in\mathbb{SL}(2)$ can be parametrized using the function $\psi:\mathbb{R}^3\to \mathbb{SL}(2)$ defined by
\begin{equation}\label{sl2-param}
\psi(s,\mu,\theta)=\cosh(s) R_1+\sinh(s) R_2\,,
\end{equation}
where 
\[
 R_1=\begin{pmatrix}
        \cos(\mu)& -\sin(\mu)\\
          \sin(\mu)& \cos(\mu)
          \end{pmatrix}\quad{\rm and }\quad 
          R_2=\begin{pmatrix}
        \cos(\theta)& \sin(\theta)\\
          \sin(\theta)& -\cos(\theta)
          \end{pmatrix}\,.
\]
\par
Let now $A=A(t)$, where $t$ belongs to some interval $I\subset\mathbb{R}$, be a curve in $\mathbb{SL}(2)$. Then there are scalar functions $s, \mu,$ and $\theta$ on $I$ such that 
\[
A(t)=\cosh(s(t)) \begin{pmatrix}
        \cos(\mu(t))& -\sin(\mu(t))\\
          \sin(\mu(t))& \cos(\mu(t))
          \end{pmatrix}+\sinh(s(t)) \begin{pmatrix}
        \cos(\theta(t))& \sin(\theta(t))\\
          \sin(\theta(t))& -\cos(\theta(t))
          \end{pmatrix}\,.
\]
\par
A straightforward calculation shows that $A^TA''$ is symmetric when
\begin{equation}\label{eq-pre-2d-1vf}
    \sinh^2 ( s ) \theta'' -\cosh^2 ( s  )  \mu''  
 +2\sinh ( s  ) \cosh ( s ) s' (
  \theta'  - \mu'   )  =0\,,
\end{equation}
which implies (after integrating)
 \begin{equation}
  \cosh^2 ( s   )  \mu'  - \sinh^2 ( s   ) \theta'  =\mathrm{constant}\,.
\label{2d-1vf}
\end{equation}
\par
Kirchhoff's solution is of this form  with $s$ constant, $\mu(t)=\mu_0t$, where $\mu_0$ is a non-zero real number, and $\theta\equiv 0$.
\par\smallskip
An easy consequence of Lemma~\ref{lem-symmetric} is the following.
\begin{theorem}\label{thm-2d}
Let $\varphi$ have the form (\ref{eq-structure-affine}), where $v$ is a diffeomorphism. Assume that $A(t)$ are parametrized via the function $\psi$ in (\ref{sl2-param}), where  $s, \mu$, and $\theta$ satisfy (\ref{2d-1vf}). Then, the function $\varphi$ provides a solution to the incompressible Euler equations.
\end{theorem}

We would like to remark that the previous theorem recovers all the solutions in \cite{CM} corresponding to what was classified as solutions of type 1 and solutions of type~2~(i) in that paper. Of course, the harmonicity condition on the coordinates of $v$ is not assumed in our case.
\par\smallskip
There is an interesting connection to the geodesics that we now describe (see also \cite{A} and \cite{AK}). Recall that  the group of orthogonal matrices $\mathbb{O}(2)$ has two disjoint components. Considering the Riemannian metric induced by the embedding $\mathbb{O}(2)\subset\mathbb{R}^4$ one can check that both components are in fact isometric (up to a constant factor) to the unit circle with the standard embedding $S^1\subset\mathbb{R}^2$. 

Now let us look for ``pure'' rotations  among curves  $A=A(t)=\psi(s(t),\mu(t), \theta(t))$ in $\mathbb{SL}(2)$ such that $s$, $\mu$, and $\theta$ satisfy (\ref{2d-1vf}), that is, we assume that the matrices $A\in \mathbb{SO}(2)$. Then, necessarily $s\equiv 0$ and hence $\mu'$ is constant. We can interpret this observation  by saying that the curve $t\to A(t)$ must be a geodesic in this case.

In the same way one could think about geodesics in $\mathbb{SL}(2)$. We consider $\mathbb{SL}(2)$ as a Riemannian manifold where the metric is induced by the embedding $\mathbb{SL}(2)\subset\mathbb{R}^4$ and  $A$ is a curve in $\mathbb{SL}(2)$.  We then obtain the last result in this section.
\begin{theorem}\label{thm-2d-2} 
If $A=A(t)$ is a geodesic in $\mathbb{SL}(2)$, then the function $\varphi(t, \alpha)=A(t)v(\alpha)$, where $v$ is a diffeomorphism, provides a solution to the incompressible Euler equations.
\end{theorem}
\begin{proof} Let us use (\ref{eq-Christoffel}) to get the geodesic equations in terms of  the parametrization in (\ref{sl2-param}), which  are
\begin{equation}\label{eq-geodesics}
\begin{cases}
2\cosh(2s) s'' +\sinh(2s)\big(  2 (s')^2-(\theta')^2-(\mu')^2\big)=0\,,\\
    \cosh(s)\mu''+2\sinh(s)s'\mu' = 0\,,\\
    \sinh(s)\theta''+2\cosh(s)s'\theta' = 0\,.
   \end{cases}
\end{equation}
Therefore, (\ref{eq-pre-2d-1vf})  and hence (\ref{2d-1vf}) hold. A direct application of Theorem~\ref{thm-2d} ends the proof.
\end{proof}
Note that, however, not all matrices $A=A(t)$ for which (\ref{2d-1vf}) holds satisfy (\ref{eq-geodesics}), so that they are not necessarily geodesics.

\section{New solutions, II}\label{sec-sol2}
In this section we consider  those solutions to the $2$-dimensional incompressible Euler equations for which the labelling map $\varphi$ is of the form (\ref{eq-structure-affine-combination}), again without any assumption on harmonicity of the vector fields involved in this description. That is, we assume that 
\begin{equation}\label{yrite2d}
\varphi(t, \alpha)=M_1(t)v(\alpha)+M_2(t)w(\alpha)\,,
\end{equation}
where
\[
 M_1(t)=\begin{pmatrix}
        \cos(\mu t)&-\sin(\mu t)\\
        \sin(\mu t)&  \cos(\mu t)
        \end{pmatrix}\quad{\rm and }\quad
         M_2(t)=\begin{pmatrix}
        \cos(\theta t)&-\sin(\theta t)\\
        \sin(\theta t)&  \cos(\theta t)
        \end{pmatrix}
\]
for certain real constants $\mu$ and $\theta$ (with $\mu\neq \theta$) and  sufficiently smooth vector fields $v$ and $w$.
\par
\begin{rk}
Notice that the matrices $M_1$ and $M_2$ are not just any curves of $\mathbb{O}(2)$ but actually geodesics in the sense explained above.
\end{rk}
 \par
The following important example as well as the solutions obtained in \cite{AY} and those classified as solutions of type 2 (ii) in \cite{CM} are of this form.
\par\smallskip
\begin{ex}\label{ex-G}
Gerstner's flow \cite{G} is of the form (\ref{yrite2d}) with $M_1=I$, $v(\alpha)=\alpha$,
\[
    M_2=\begin{pmatrix}
        \cos(k t)& -\sin(kt)\\
          \sin(kt)& \cos(kt)
          \end{pmatrix}\,,\quad {\rm and}\quad
          w(\alpha)=
           \frac{e^{k\alpha_2}}{k}
          \begin{pmatrix}
        \sin(k\alpha_1)\\
          -\cos(k\alpha_1)
          \end{pmatrix}\,,
\]
where $k$ is a non-zero real number.
\end{ex}
\par
We should mention that, in contrary to those cases considered in the previous section, where the (global) injectivity of the vector field $v$ implies the global injectivity of the functions $\varphi$, now some extra conditions must be satisfied in order to guarantee the injectivity of the functions in (\ref{yrite2d}) in a similar way as what occur in the known solutions of this type  (see, for instance, \cite[Thm. 4]{CM}). With the hope to make our exposition more clear, we have decided to make a local analysis in the proofs of our main results in this section and do not include the details on the conditions for the global univalence of the mappings in (\ref{yrite2d}).
\par\smallskip
Notice that the vector fields $v$ and $w$ in the previous example satisfy that their coordinates are harmonic, though there are other solutions with the same structure as in (\ref{yrite2d}) which are not harmonic, as our next theorem shows. 
\begin{theorem} \label{thm-pde}
Let the map $\varphi$ have the form (\ref{yrite2d}). Then $\varphi$ provides a solution to the $2$-dimensional Euler equations if $v$ and $w$ satisfy the system of PDEs
\begin{equation}\label{eq-system-pde}
\begin{cases}
  w^1_{10}v^2_{01}-w^1_{01}v^2_{10}-w^2_{10}v^1_{01}+w^2_{01}v^1_{10}=0\,,\\
w^2_{10}v^2_{01}-w^2_{01}v^2_{10}+w^1_{10}v^1_{10}-w^1_{01}v^1_{10}=0\,,
\end{cases}
\end{equation}
and $\det(dv)+\det(dw)\ne0$. 
\label{2d-2vk-rot}
\end{theorem}

Before stating the proof of Theorem~\ref{thm-pde},  let us see how this result relates to the previously known solutions and make some comments. 
\par
Perhaps one can say that Theorem \ref{2d-2vk-rot} explains why harmonic maps have played a prominent role in the analysis of Euler equations. If we choose $v$ as a CR map (so that, in particular, the coordinates of $v$ are harmonic functions), then the system in Theorem \ref{2d-2vk-rot} becomes
\[
     \begin{cases}
v^1_{10}\big(  w^1_{10}-w^2_{01}\big)+v^1_{01}\big(w^1_{01}+w^2_{10}\big)=0\,,\\
v^1_{10}\big(w^1_{01}+w^2_{10}\big)+v^1_{01}\big(w^1_{10}-w^2_{01}\big)=0\,.
\end{cases}
\]
Hence  $w$ must be a CR map with the sign convention of complex analysis. These are precisely the solutions found in \cite{AY}, and this choice of $v$ and $w$ is called solutions of  type 2 (ii) in \cite{CM}.   

However, there are plenty of solutions of (\ref{eq-system-pde}) which are not harmonic: we have 4 unknowns and only 2 equations so that  we can, for example, choose $w$ arbitrarily and then solve the corresponding system, which can be written in the form $Lv=0$, for $v$. In fact the two linear  equations obtained in this way are rather special, as the following result shows. 
\begin{lemma} If $\det(dw)\ne0$, then $L$ is elliptic.
\end{lemma}
\begin{proof} To show that $L$ is elliptic, we need to check that its principal symbol $\sigma L$ defined by (\ref{eq-elliptic}) is injective for all $\xi\in\mathbb{R}^2$, $\xi\neq 0$.
It suffices, then, to check
\[
    \det(\sigma L)=\big(w^1_{01}\xi_1-w^1_{10}\xi_2)^2+
    \big(w^2_{01}\xi_1-w^2_{10}\xi_2)^2\ne0
\]
for all real non-zero vectors $\xi$, which holds if $\det(dw)\ne0$.
\end{proof}
When considering elliptic boundary value problems, one needs to impose appropriate boundary conditions. Since $\det(\sigma L)$ is a second order polynomial, we need to impose one boundary condition. The relevant criterion is known as \emph{Shapiro-Lopatinskij condition} \cite{agranovich}, which is easy to check in our case since rather amazingly we have the following explicit factorization
\begin{eqnarray*}
    \det(\sigma L)&=&\frac{1}{|w_{10}|^2}\,\big( |w_{10}|^2\xi_2-(\langle w_{10}, w_{01}\rangle+i\,\det(dw))\xi_1\big)\\
    &\times& \big( |w_{10}|^2\xi_2-(\langle w_{10}, w_{01}\rangle-i\,\det(dw))\xi_1\big)\,.
\end{eqnarray*}
\par\medskip
\begin{rk}
In (\ref{yrite2d}) we have assumed that  both matrices $M_1$ and $M_2$ are rotations. We could as well choose one or both of the matrices to be reflections. The computations would be precisely the same as above, except a few changes in the signs. The overall conclusion is the same: given one vector field one obtains an elliptic system for the other.
\end{rk}
\par\smallskip
Let us then turn to the proof of the theorem. To this end,  notice that the statement of Theorem~\ref{thm-pde} as well as its conclusion already make it clear that it is not essential to consider  $v$ and $w$ as distinct objects, it is more natural to consider a map with four components. Hence let us introduce a more convenient way to analyze the problem. We will look for solutions in the following form:
\begin{equation}
\varphi(t, \alpha)=A(t)u(\alpha)\,,
\label{yrite2da}
\end{equation} 
where $A=(M_1\, |\, M_2)\,:\,\mathbb{R}\to\mathbb{R}^{2\times 4}$ is a $1\times 2$ block matrix with first entry $M_1$ and second entry $M_2$ and $u\,:\,\mathbb{R}^2\to\mathbb{R}^{4}$ is the vector $u=(v\, |\, w)=(v^1, v^2, w^1, w^2)$.
      
\par
By the Cauchy-Binet formula we can write
\begin{equation}
  \det(d\varphi)=\sum_{i=1}^6 p_ig_i\,,
\label{det-CB}
\end{equation}
where $p_i$ are the $2\times 2$ minors of $A$ and $g_i$ are the $2\times 2$ minors of $du$. 
\par
Let now $B=A^TA''$ and let $y=d\varphi^T\varphi$. If the map $\varphi$ given by (\ref{yrite2da})  solves Problem~\ref{problem}, then we must have
\[
   N=dy-dy^T=0\,.
\]
Note that $N$ is skew symmetric so we have only one condition: $N_{12}=-N_{21}=0$. Let us analyze further this condition. 

First, let us  write $N$ with index notation of Riemannian geometry, where it matters if indices are up or down (though in our context we can simply arrange the expressions so that the summation convention works). Let us use the notation $y_i=V^\ell_{;i}B_{\ell k}V^k$ to get
\begin{align*}
       N_{ij}=&y_{i;j}-y_{j;i}=u^\ell_{;i}B_{\ell k}u^k_{;j}-u^\ell_{;j}B_{\ell k}u^k_{;i}\\
         =&B_{\ell k}\big(u^\ell_{;i}u^k_{;j}-u^\ell_{;j}u^k_{;i}\big)=
            u^\ell_{;i}\big(B_{\ell k}-B_{k\ell }\big) u^k_{;j}\,.
\end{align*}
A straightforward calculation shows that since the parameters $\mu$ and $\theta$ that determine the matrix $A$ are different, then the matrix $B$ is not symmetric (which would imply $N=0$). We can then write $N_{12}$ as
\begin{equation}
     N_{12}=\sum_{i=1}^6 f_i g_i=B_{\ell k}\big(u^\ell_{;1}u^k_{;2}-u^\ell_{;2}u^k_{;1}\big)\,,
\label{diff-yht-N}
\end{equation}
where $f_i$ depend only on $t$ and $g_i$ are, \emph{again}, the $2\times 2$ minors of $du$. 
\par\smallskip
The strategy of the proof of Theorem~\ref{thm-pde} is now clear: we compute  the formulas (\ref{det-CB}) and (\ref{diff-yht-N}) -the use of symbolic computation programs  is fundamental to this end- and then try to choose $A$ and $u$ such that $ \det(d\varphi)$ does not depend on time and $N_{12}$ is identically zero. Note that a priori we have $12$ conditions and $12$ independent functions. However, as we will see below, there are plenty of possibilities to choose the appropriate functions.

\begin{proof} (Theorem \ref{2d-2vk-rot})
We now choose $A=(M_1\, |\, M_2)$ in (\ref{yrite2da}),  where $M_1$ and $M_2$ are two given matrices as in (\ref{yrite2d}). Let us introduce the variables
$ c_1=\cos(\mu t)$, $s_1=\sin(\mu t)$, $c_2=\cos(\theta t)$, and $s_2=\sin(\theta t)$; and the ideal
\[
      I=\langle c_1^2+s_1^2-1,c_2^2+s_2^2-1\rangle\subset\mathbb{A}=\mathbb{Q}(\mu,\theta)[c_1,s_1,c_2,s_2]\,,
\]
where $\mathbb{Q}(\mu,\theta)$ is the field of rational functions of $\mu$ and $\theta$. 
\par
Let us first use the division algorithm to compute
\[
         \det(d\varphi)=\sum_{i=1}^6 p_ig_i\,,
\]
where $p_i\in\mathbb{A}$ and $g_i$ are the $2\times 2$ minors of $du$. Then, we compute the normal forms of $p_i$ with respect to $I$, which gives
\[
\det(d\varphi)=u^1_{10}u^2_{01}-u^1_{01}u^2_{10}+u^3_{10}u^4_{01}-u^3_{01}u^4_{10}+
(c_1c_2+s_1s_2)q_1+(s_1c_2-s_2c_1)q_2\,,
\]
where
\[
q_1=u^3_{10}u^2_{01}-u^3_{01}u^2_{10}
-u^4_{10}u^1_{01}+u^4_{01}u^1_{10}\quad{\rm and}\quad
q_2=u^4_{10}u^2_{01}-u^4_{01}u^2_{10}
+u^3_{10}u^1_{01}-u^3_{01}u^1_{10}\,.
\]
Therefore, we see that $\det(d\varphi)$ is independent of $t$ if  $q_1=q_2=0$. 
\par
Moreover, by  dividing the polynomial $N_{12}$ in (\ref{diff-yht-N}) with respect to $\langle q_1,q_2\rangle$ gives
\[
   N_{12}=\big(\mu^2-\theta^2\big)\big((s_2c_1-c_2s_1)q_1+(c_1c_2+s_1s_2)q_2\big)=0
\]
if $q_1=q_2=0$. The condition that  $q_1=q_2=0$ is precisely (\ref{eq-system-pde}) since $u=(v | w)$. This ends the proof.
\end{proof}
\par\medskip
Instead of considering matrices $A=(M_1\, |\, M_2)$, where $M_1$ and $M_2$ are as in (\ref{yrite2d}),  we could consider more general $2\times 4$ matrices $A=(a_{ij})$ in (\ref{yrite2da}). The systematic analysis of all cases would take too much space so we will simply give a particular construction which indicates the general idea and produces a result which we hope can be of some interest to the reader.
\par
We just impose to the matrix  $A$ in (\ref{yrite2da})  the following constraints:
\begin{equation}\label{eq-constraintA}
  \begin{cases}
     a_{11}a_{22}-a_{12}a_{21}-1=0\,,\\
     a_{13}a_{24}-a_{14}a_{23}-1=0\,,\\
     a_{14}-a_{12}+a_{11}=0\,,\\
     a_{24}-a_{22}+a_{21}=0\,.
     \end{cases}
\end{equation}
The first two conditions in (\ref{eq-constraintA}) are rather natural since one suspects that $\mathbb{SL}(2)$ will anyway be essential in the analysis and they are obviously satisfied by any matrix $A=(M_1\, |\, M_2)$ as above. The other two conditions  in (\ref{eq-constraintA}) are not a priori clear, though the 4 conditions together lead to the following family of solutions to the $2$-dimensional Euler equations. 

\begin{theorem} \label{thm-pde2}
Assume that the $2\times 4$ matrix $A=(a_{ij})$ solves the system
\begin{equation}
\begin{cases}
     a_{11}a_{22}-a_{12}a_{21}-1=0\,,\\
     a_{13}a_{24}-a_{14}a_{23}-1=0\,,\\
     a_{14}-a_{12}+a_{11}=0\,,\\
     a_{24}-a_{22}+a_{21}=0\,,\\
    a_{21}''a_{22}-a_{22}''a_{21}+a_{11}''a_{12}-a_{12}''a_{11}=0\,,
\end{cases}
\label{aikasys}
\end{equation}
and suppose that the vector $u:\mathbb{R}^2\to \mathbb{R}^4$ is a solution of
\begin{equation}
\begin{cases}
 -u^3_{10}u^2_{01}+u^3_{01}u^2_{10}-u^3_{10}u^1_{01}+u^3_{01}u^1_{10}=0\,,\\
    -u^4_{10}u^3_{01}+u^4_{01}u^3_{10}-u^3_{10}u^1_{01}+u^3_{01}u^1_{10} =0\,.
\end{cases}
\label{transport}
\end{equation}
Then (\ref{yrite2da}) produces a solution to the $2$-dimensional incompressible Euler equations.
\end{theorem}
\begin{proof}
Let us introduce the ideal
\begin{align*}
    I=&\langle  a_{11}a_{22}-a_{12}a_{21}-1,     a_{13}a_{24}-a_{14}a_{23}-1,
     a_{14}-a_{12}+a_{11},     a_{24}-a_{22}+a_{21}\rangle\\
     \subset&\ \mathbb{A}=
     \mathbb{Q}[a_{11},a_{12},a_{13},a_{14},a_{21},a_{22},a_{23},a_{24}]\,.
\end{align*}
Using again the division algorithm we have
\[
         \det(d\varphi)=\sum_{i=1}^6 p_ig_i\,,
\]
where $p_i\in\mathbb{A}$ and $g_i$ are, again, the $2\times 2$ minors of $du$. The computation of the normal forms of $p_i$ with respect to $I$ gives
\[
       \det(d\varphi)=g_1+g_2+g_3-g_4+p_5(g_4+g_5)+g_6\,.
\]
Hence, $ \det(d\varphi)$ is independent of $t$ if we require $g_4+g_5=0$, which is precisely the first equation in (\ref{transport}) since
\[
g_4+g_5=-u^3_{10}u^2_{01}+u^3_{01}u^2_{10}-u^3_{10}u^1_{01}+u^3_{01}u^1_{10}\,.
\]
\par
Let us now identify the matrix $A$ (\emph{resp.} $A''$) with the vector $a$ (\emph{resp.} $a''$) which contains all the elements of $A$ and consider the ideal 
\[
      I_1=I+\langle  a_{14}''-a_{12}''+a_{11}'',     a_{24}''-a_{22}''+a_{21}''\rangle\\
     \subset\ \mathbb{A}_1=\mathbb{Q}[a,a'']\,.
\]
Next, we compute the formula (\ref{diff-yht-N}) which yields
\[
       N_{12}=  \sum_{i=1}^6 f_i g_i\,,
\]
where $f_i\in\mathbb{A}_1$. Computing the normal forms we find
\[
        \mathsf{NF}(f_2,I_1)= \mathsf{NF}(f_3,I_1)= \mathsf{NF}(f_6,I_1)=
          a_{21}''a_{22}-a_{22}''a_{21}+a_{11}''a_{12}-a_{12}''a_{11}\,.
\]
Then we define a bigger ideal
\[
    I_2=I_1+\langle   a_{21}''a_{22}-a_{22}''a_{21}+a_{11}''a_{12}-a_{12}''a_{11}\rangle
\]
and reduce $N_{12}$ with respect to $I_2$, which yields
\[
  N_{12}=\hat f_1g_1+\hat f_4 g_4+\hat f_5 g_5\,,
\]
where  $\hat f_1+\hat f_4-\hat f_5=0$ . Now using the relation  $g_4+g_5=0$, we get $N_{12}=\hat f_1(g_1-g_4)$. Hence, the assumption that $g_1-g_4=0$ (which is equivalent to the second PDE for $u$ in (\ref{transport})) gives $f=0$.
\end{proof}

Notice that in the system (\ref{transport}) there are two equations for four unknowns. We can then, for instance, consider two given arbitrary functions $u^1$ and $u^3$ to obtain a decoupled system for $u^2$ and $u^4$. This resulting system for $u^2$ and $u^4$ is a kind of transport equation which can be solved with the method of characteristics in the usual way. (In particular, again the harmonicity or lack of harmonicity of $u$ plays no role here.)   
\par
The system (\ref{aikasys}) is also underdetermined so that there are also plenty of possibilities for choosing the time dependence of the system. For instance, we can choose the submatrix 
\[
        \hat A=\hat A(t) =\begin{pmatrix}
        a_{11}&a_{12}\\
        a_{21}& a_{22}
        \end{pmatrix}\in \mathbb{SL}(2)
\]
such that $\hat A$ is a geodesic. Then the first and last equations in  (\ref{aikasys}) are automatically satisfied.  Then, we get $(a_{14},a_{24})$ from the third and fourth equations given in (\ref{aikasys}) and, finally, we just need to consider the linear relation $ a_{13}a_{24}-a_{14}a_{23}-1=0$ to obtain $(a_{13},a_{23})$ and produce a $2\times 4$ matrix $A$ that solves the system  (\ref{aikasys}).

\end{document}